\documentclass[reqno,10pt]{amsart}
\usepackage{xspace,amsmath}
\usepackage[bookmarksnumbered,colorlinks]{hyperref}
\usepackage{graphics}
\usepackage[english]{babel}
\usepackage[T1]{fontenc}
\usepackage{enumerate}
\usepackage{graphics}
\def\bb#1\eb{\textcolor{black}
	{#1}} %
\usepackage[normalem]{ulem}

\newcommand{\bt}{\begin{theorem}}                     
	\newcommand{\et}{\end{theorem}}                       
\newcommand{\bd}{\begin{definition}}                  
	\newcommand{\ed}{\end{definition}}                    
\newcommand{\bl}{\begin{lemma}}                       
	\newcommand{\el}{\end{lemma}}                                   
\newcommand{\bpr}{\begin{proposition}}                  
	\newcommand{\epr}{\end{proposition}}                    
\newcommand{\bere}{\begin{remark}}                      
	\newcommand{\ere}{\end{remark}}                         
\newcommand{\beq}{\begin{equation}}
	\newcommand{\eeq}{\end{equation}}
\def\bal#1\eal{\begin{align}#1\end{align}}              
\def\baln#1\ealn{\begin{align*}#1\end{align*}}          
\def\bml#1\eml{\begin{multline}#1\end{multline}}        
\def\bmln#1\emln{\begin{multline*}#1\end{multline*}}  
\def\bga#1\ega{\begin{gather}#1\end{gather}}
\def\bgan#1\egan{\begin{gather*}#1\end{gather*}}
\newcommand{\de}{\mathrm{d}}                        
\DeclareMathOperator{\cat}{cat}                     
\newcommand{\N}{\ensuremath{\mathbb{N}}\xspace}     
\newcommand{\R}{\ensuremath{\mathbb{R}}\xspace}     
\newcommand{\eps}{\varepsilon}                      
                   
\newcommand{\inte}{\int_0^1\!\!}
\newcommand{\Vn}{\mathcal{V}}

\newcommand{\D}{\ensuremath{\mathcal{D}}}
\newcommand{\A}{\ensuremath{\mathfrak{A}}\xspace}

\newtheorem*{theorem*}{Theorem}
\newtheorem{theorem}{Theorem}[section]
\newtheorem{corollary}[theorem]{Corollary}
\newtheorem{lemma}[theorem]{Lemma}
\newtheorem{proposition}[theorem]{Proposition}
\theoremstyle{definition}
\newtheorem{definition}[theorem]{Definition}
\theoremstyle{remark}
\newtheorem{remark}[theorem]{Remark}
\theoremstyle{remark}

\theoremstyle{definition}

\title[Multiple connecting geodesics of  a Randers-Kropina metric]%
{Multiple connecting geodesics of a   Randers-Kropina metric via homotopy theory for solutions of an affine control system}

\author[E. Caponio]{Erasmo Caponio}
\address{Dipartimento di Meccanica, Matematica  e Management, Politecnico di Bari,  Bari, Italy}
\email{erasmo.caponio@poliba.it}
\thanks{EC and AM are partially supported by E.C. and A.M are partially supported by European Union - Next Generation EU - PRIN 2022 PNRR {\it ``P2022YFAJH Linear and Nonlinear PDE's: New directions and Application''}, and GNAMPA INdAM – Italian National Institute of High Mathematics. This work was also partially supported by the Italian Ministry of University and Research under the Programme “Department of Excellence” Legge 232/2016 (Grant No. CUP - D93C23000100001). }

\author[M. A. Javaloyes]{Miguel Angel Javaloyes}
\address{Departamento de Matem\'aticas, 	Universidad de Murcia,
	Murcia, Spain}
\email{majava@um.es}
\thanks{ MAJ was supported by  project PGC2018-097046-B-I00 funded by MCIN/ AEI /10.13039/501100011033/ FEDER ``Una manera de hacer Europa''.}

\author[A. Masiello]{Antonio Masiello}
\address{Dipartimento di Meccanica, Matematica  e Management, Politecnico di Bari, Bari, Italy}
\email{antonio.masiello@poliba.it}



\dedicatory{Dedicated to the memory of Edward Fadell and Sufian Husseini}

\begin{document}

	\begin{abstract}
		We consider  a geodesic problem in a manifold endowed with a  Randers-Kropina metric. This is a type  of  singular   Finsler metric arising both in the description of the lightlike vectors of a  spacetime endowed with a causal Killing vector field and in the Zermelo's navigation problem with a wind represented by a vector field having norm not greater than one. By using Lusternik-Schnirelman theory, we prove existence of infinitely many geodesics between two given points when the manifold is not contractible.   Due to the type of nonholonomic constraints that the velocity vectors must satisfy, this is achieved thanks to \bb a recent result about the homotopy type of the set of solutions of  an affine control system with a controlled drift and related to   a  corank one, completely nonholonomic distribution of step $2$.  \eb \end{abstract}
	\maketitle
	
	\section{Introduction}\label{intro}
	Randers-Kropina metrics emerge  as  particular cases of  {\em wind Finslerian structures},  a generalization of the notion of a Finsler metric introduced in \cite{CaJaSa14},  that consists in considering on a finite dimensional manifold  a smooth field of  strongly convex compact hypersurfaces  ({\em indicatrices})  which may not  include  pointwise the zero vector. A Randers-Kropina metric corresponds to the case where every indicatrix is an ellipsoid and the zero vector belongs to its bounded region (including the ellipsoid),   but it is not outside. 
	When the indicatrices are ellipsoids  these  geometric  structures can be   interpreted dynamically in two different ways.   
	
	The first one is related to the description of the lightlike vectors in a spacetime endowed with a Killing vector field without any causal character. In the particular case that  the Killing vector field is causal, a Randers-Kropina metric  is the Lagrangian that associates   with  the spatial component of any lightlike vector its time component  (see Section~\ref{causalKilling} for details). 
	
	The second one is related  to the {\em Zermelo's navigation problem} which consists in finding the paths between two points $x_0$ and $x_1$ that minimize the travel time of a ship or an airship moving under  a wind in a Riemannian manifold $(S,g_0)$ (see \cite{Zermel31, Carath67} for the original formulation of the problem in the Euclidean space $\R^2$). If the wind is time-independent then it can be represented  by a vector field $W$ on $S$. In the particular case when  $g_0(W,W)\leq 1$,  the solutions of the problem  are the pregeodesics of a Randers-Kropina metric, associated with the  data $g_0$ and $W$, which are minimizers of its length functional, see \cite[Corollary 6.18]{CaJaSa14}.\footnote{Actually the setting in \cite{CaJaSa14} is more general as a vector field $W$ with no restriction on its  $g_0$-norm is considered} 
	
	Both interpretations enlighten  a characteristic of  these metrics: they can be viewed as  limit singular Lagrangians $F$  of regular strongly convex ones $F_\eps$ (which are Randers metrics) where, as $\eps\to 0$, the strongly convex indicatrices  at $x$, $\{v:F_\eps(x,v)=1\}$,  containing the zero vector inside, deform up to the strongly convex hypersurface $\{v:F(x,v)=1\}$, which might be, at some $\tilde x$,  a punctured hypersurface with the zero vector as a puncture. Thus, the  singularity of these metrics, at those points $\tilde x$,  is related to the fact that they are not well-defined at vectors in the  tangent hyperplanes at zero to the  indicatrix.   These hyperplanes  can be identified   with the kernel at the point $\tilde x$    of a one-form $\omega$ globally defined. Assuming that $\omega$ does not vanish at any point and  taking its kernel at each point, we get   a globally defined smooth distribution,  with constant rank,   that we  assume to be \bb nowhere integrable, i.e $(\omega\wedge\de \omega)_p\neq 0$, for all $p\in S$. \eb A translation of this distribution    defines an affine control system  whose  solutions between   two points $x$ and $y$,  give  a family of  curves having  velocity vectors in the regular set of the Randers-Kropina metric $F$.
	\bb The homotopy type of the space of solutions between two points  of an affine control system   has been studied in \cite{DomRab12} with respect to the $W^{1,1}$-topology and in \cite{BoaLer17} in the $W^{1,p}$-topology, for $p\in [1, p_c)$ where  $p_c\leq 2$. In both papers, whenever the underlying distribution is completely nonholonomic, the space of solutions considered  is proved to be homotopy equivalent to the based loop space (with the compact-open topology).  In \cite{CaMaSu24}  such results have been extended to a subset of the solutions of a differential inclusion defined by a field of open half-spaces (with the union of the zero-section) whose boundary, in each tangent space, is  a corank one, nowhere integrable distribution. Remarkable,  the homotopy type of this set of solutions is obtained in the $W^{1,p}$-topology, for any $p\in [1, +\infty)$  (see Theorem~\ref{lerario}).  A celebrated result of E. Fadell and S. Husseini \cite{FadHus91} implies then the existence  of  compact subsets  in  this  space  of solutions   (with the $W^{1,2}$-topology) of arbitrarily high Lusternik-Schnirelman category, provided that the manifold where the curves assume values is non-contractible. 
	We show  that the energy functional of the Randers-Kropina metric is bounded on these compact subsets (see Proposition~\ref{bounded}) though it is not necessarily continuous in the $W^{1,2}$-topology.  \eb This fact allows us to control,  uniformly with respect to $\eps$,   the critical values of the energy functionals of the approximating regular metric $F_\eps$, obtained by  minimax, and then to get  infinitely many critical values for the energy functional of the Randers-Kropina  metric $F$, see Theorem~\ref{multiplegeos}.
	
	Due to the singularities of the metric $F$, establishing the convergence of critical points  of the energy functionals of $F_\eps$ to a critical point of the energy functional of $F$ is not a straightforward task. We notice that  
	the approximating   Randers metrics $F_\eps$ can be obtained as the Fermat metrics associated with one-dimensional higher  Lorentzian standard stationary spacetimes $(S\times\R, g_\eps)$ by stationary-Randers correspondence, see \cite{CaJaMa11,CaJaSa11}  and the review paper \cite{Masiel21}. The metrics of these  spacetimes are  perturbations of the Lorentzian metric  associated with the Randers-Kropina one, where the vector field $\partial_t$, $t$ the natural coordinate of the factor $\R$, is Killing and timelike. A key point is that, differently from the metrics $F_\eps$ and $F$, the Lorentzian metrics $g_\eps$ approximate $g$ smoothly w.r.t. $\eps\to 0$, providing a good setting for studying convergence of geodesics (see Remark~\ref{smoothdependence} and  Proposition~\ref{energies}),   thanks to a  correspondence existing between future-pointing lightlike pregeodesics of the overlaying Lorentzian metrics and pregeodesics of the underlying Finsler metrics (see Theorems~\ref{fermat}).  Our first convergence result (Proposition~\ref{preconv}) concerns a sequence of future-pointing lightlike geodesics $\{\gamma_n\}$ with bounded arrival times (each $\gamma_n$ is a geodesic of $g_{\eps_n}$, $\eps_n\to 0$) and it is based on Lemma~\ref{convexforall} that is  of independent interest  since we prove the existence of convex neighborhoods for all the elements of a family of  $C^2$  pseudo-Riemannian metrics.  The approximation with standard stationary spacetimes   has been recently used in \cite{BaCaFl17} to prove geodesic connectedness of a globally hyperbolic spacetime endowed with a complete lightlike Killing vector field and in \cite{CaGiMS21} to obtain results about existence and multiplicity of geodesics of a Kropina metric. 
	In \cite{CaGiMS21},  geodesics are obtained as   minimizers of the  length functional on the connected components of the considered paths space assuming that  this contains at least one admissible curve. 
	Using Proposition~\ref{preconv}, the  results  in \cite{CaGiMS21} extend  to the setting of the present paper (see Remark~\ref{minimo} for the existence of a minimizer between two given points).
	Thus, our aim in this work  is   to look for geodesics which are not necessarily minimizers, by using  Lusternik-Schnirelman  theory.

	\section{Randers-Kropina metrics and spacetimes with a causal Killing vector field}\label{causalKilling}
	In this section we introduce  Randers-Kropina metrics  and describe their connection with lightlike vectors in a spacetime admitting a causal Killing vector field.   The more specific case of a timelike Killing vector field was studied in \cite{CaJaMa11, CaJaSa11} while the general case of a Killing field with no causal character was studied in \cite{CaJaSa14}; the reader can  look at Section 4 in \cite{CaJaSa14} for  further  properties about the  case considered in the present paper. 
	
	Let $S$ be 	a connected smooth manifold of finite dimension $m\geq 2$, endowed with a smooth Riemannian metric $g_0$, a smooth one-form $\omega$ and a smooth function $\Lambda$. Let $TS$ be the tangent bundle of $S$ and $\pi:TS\to S$ be the canonical projection.  We consider on $S\times \R$ the bilinear symmetric tensor 
	\beq\label{g}g=g_0+\omega\otimes \de t +\de t\otimes \omega-\Lambda \de t^2,\eeq
	where $t$ is the natural coordinate associated with the factor $\R$. It can be easily seen that $g$ is a Lorentz metric if and only if
	\begin{equation}
		\label{lorentzian} \Lambda(x) +\|\omega\|^2_{x}>0,\quad \text{for all $x\in S$,}
	\end{equation}
	being $\|\omega\|_x$ the $g_0$-norm of $\omega$ in $T_xS$.
	Notice that the vector field $\partial_t$ is a complete Killing vector field for $g$, and $\Lambda(x)=g_{(x,t)}(\partial_t,\partial_t)$, for all $t\in\R$, so that the value of $\Lambda(x)$ defines the causal character of $\partial_t$ at the points $(x,t)$, for all $t\in\R$.  
	
	We recall that 
	the opposite of the gradient of the temporal function  $(x,t)\in S\times\R\mapsto t\in\R$ gives a  time-orientation to $S\times \R$ in the sense that it allows us to choose, continuously and  globally, one of the two causal cones at $T_p(S\times \R)$, $p\in S\times \R$. The selected ones  constitute the set of future-pointing causal vectors in $T(S\times \R)$; with our convention on the signature of the metric $g$,  they are non-zero  vectors $w\in T(S\times \R)$, such that $g(w,w)\leq 0$ and $\de t(w)> 0$ so that $\partial_t$ is future-pointing as well. Thus, a causal vector $w\in T(S\times\R)$ is future-pointing if and  only if $g(w,\partial_t)\leq 0$.  
	
	A future-pointing  lightlike vector $(v,1)\in T_{(x,t)}(S\times \R)$ satisfies, by definition, the equation 
	
	\beq\label{ellipsoid}g_{0}(v,v)+2\omega(v)-\Lambda(x)=0,\eeq
	which, thanks to \eqref{lorentzian}, represents   a non-degenerate  ellipsoid $\Sigma_x$ in $T_x S$, for each $x\in S$. We call the field $x\in S\mapsto \Sigma_x$ a 
	{\em wind Riemannian structure}.  The other way round, it can be proved  then  that any wind Riemannian structure can be associated with a conformal class of spacetimes of the above type (see \cite[Theorem 3.10]{CaJaSa14}). 
	
	\smallskip
	In this work we assume that $\Lambda(x)\geq 0$ and   $\omega_x\neq 0$, for all $x\in S$. 
	\smallskip
	
	As $\Lambda\geq 0$, the Killing vector field $\partial_t$ is causal.
	
	Notice that  a  vector $v\in TS$ satisfies 
	\eqref{ellipsoid} (in other words,  $(v,1)$ is lightlike for $g$) if and only if  
	\[1=\begin{cases}
		R(v):=\dfrac{1}{\Lambda(x)}\left(\omega(v)+\sqrt{\Lambda(x)g_0(v,v)+\omega^2(v)}\right)&\text{if $\Lambda(x)>0$}\\
		K(v):=-\dfrac 1 2\dfrac{g_{0}(v,v)}{\omega(v)}&\text{if $\Lambda(x)=0$,}
	\end{cases}\]
	where $x=\pi(v)$ (see \cite[Proposition~3.12]{CaJaSa14}).
	The functions $R$ and $K$ on the right-hand side above define two  types  of Finsler norms  on $T_x S$. It is well-known that  $R$ is a standard Minkowski  norm of Randers type  defined on the whole tangent space  (see \cite{CaJaMa11}).  On the other hand, $K$  is a Kropina norm, i.e. a  conic Minkowski norm in the language of \cite{JavSan14}, which is positive only on  $\{v\in T_xS: \omega(v)<0\}$ and whose  
	fundamental tensor   
	\[\frac{\partial^2}{\partial t\partial s} \frac{1}{2}K^2(v+tu+sw)_{|t=s=0},\]
	$u,w\in T_{x}S$,
	is positive definite outside the kernel of $\omega$  (see \cite[Corollary 4.12]{JavSan14}).
	
	Following \cite{CaJaSa14}, these two norms can be unified  by considering the function  $F$ defined as:
	\beq
	F(v):=\frac{g_{0}(v,v)}{-\omega(v)+\sqrt{\Lambda g_{0}(v,v)+\omega^2(v)}},\label{randerskropina}\eeq
	whose domain is the subset $\mathcal A$ of $TS$ defined as 
	\[\mathcal A:=\left\{v\in TS:\begin{array}{ll} v\in T_{\pi(v)}S&\text{if $\Lambda(\pi(v))>0$}\\
		\omega(v)<0&\text{if $\Lambda(\pi(v))=0$}\end{array}\right\};\]
	$F$  is then called a {\em Randers-Kropina metric}. 
	
	Notice that for all $x\in S$, with $\Lambda(x)=0$, $0\not\in  \mathcal A_x:=\mathcal A\cap T_xS$ though it is an accumulation point of the set $\{v\in T_x S: F(v)=1\}$   (clearly, $\Sigma_x=\{v\in T_x S: F(v)=1\}\cup\{0\}$). 
	
	\smallskip
	Thus, at $x\in S$ such that $\Lambda(x)=0$, $F$ is  not continuously  extendible  at $0$ neither it is extendible on the vectors in $\mathcal D_x:=\ker \omega_x$.
	\smallskip
	
	Let us denote by $\Omega_{x_0x_1}(\mathcal A)$, the set of   continuous and piecewise smooth, {\em admissible} curves from $x_0$ to $x_1$,  i.e.   
	\[\Omega_{x_0x_1}(\mathcal A):=\{\gamma\colon[0,1]\to S: \gamma(0)=x_0,\ \gamma(1)=x_1,\\ \dot \gamma^-(s), \dot\gamma^+(s)\in \mathcal A,\ \forall s\in(0,1)\}\]
	(here  $\dot\gamma^-(s)$ and $\dot\gamma^+(s)$ denote respectively the left and the right derivative  of $\gamma$ at the point  $s$).
	
	We also denote  $\Omega_{x_0x_1}(TS)$, i.e. the classical space of continuous piecewise smooth
	paths between $x_0$ and $x_1$, by $\Omega_{x_0x_1}$.
	
	A geodesic of $(S,F)$  connecting $x_0$ to $x_1$,  $x_0,x_1\in S$, is by definition a critical point of the energy functional 
	\[
	E(\gamma)= \frac{1}{2}\int_0^1F^2(\dot \gamma)\de s,
	\]
	on $\Omega_{x_0x_1}(\mathcal A)$.
	Notice that, as $\mathcal A$ is an open subset of $TS$, piecewise smooth  variational  vector fields along a curve $\gamma\in \Omega_{x_0x_1}(\mathcal A)$ are well-defined and then it makes sense to define geodesics as critical points of $E$.
	Moreover since the fundamental tensor of $F$ is non-degenerate on $\mathcal A$, it can be proved that the Legendre transform of $F$ is injective (see \cite[Proposition 2.51]{CaJaSa14}) and then  a critical point $\gamma$ of $E$ is actually smooth and  parametrized with $F(\dot\gamma)=\mathrm{const.}$ (see also \cite[Lemma 2.52]{CaJaSa14}).

	Geodesics of a Randers-Kropina space $(S,F)$ are related to future-pointing lightlike geodesics of the  product spacetime $(S\times \R,g)$, $g$ as in \eqref{g}.
	Let us  introduce the bilinear symmetric tensor $h$ on $S$ defined as
	\[h:=\Lambda g_0 +\omega\otimes\omega.\]
	Notice that on the open region $S_0:=\{x\in S:\Lambda(x)>0\}$, $h$ is a Riemannian metric, being degenerate on $S\setminus S_0$. 
	Being $\partial_t$ Killing,  any geodesic  of $(S\times\R, g)$
	must satisfy  $g(\dot\gamma(s),\partial_t)=C_\gamma$, with $C_\gamma\in\R$ constant. 
	For a future-pointing lightlike geodesic $\gamma(s)=\big(\sigma(s),t(s)\big)$ of $(S\times\R, g)$ (i.e.   a geodesic  such that $g\big(\dot\gamma(s),\dot\gamma(s)\big)=0$,  for all $s$, and $C_\gamma\leq 0$) we have   $C_\gamma^2=h(\dot\sigma,\dot\sigma)$. 
	This can be easily checked, taking into account that at instants $s$ such that $\sigma(s)\in S\setminus S_0$, $C_\gamma^2=g\big(\dot\gamma(s),\partial_t\big)^2=\omega^2(\dot\sigma)=h(\dot\sigma,\dot\sigma)$, otherwise $g(\dot\gamma,\dot\gamma)=0$ is equivalent to 
	\[g_0(\dot\sigma,\dot\sigma)+\frac{\omega^2(\dot\sigma)}{\Lambda}-\left(\frac{\omega(\dot\sigma)}{\sqrt{\Lambda}}-\sqrt{\Lambda}\dot t\right)^2=0,\]
	and we get the result taking into account that $C_\gamma=\omega(\dot\sigma)-\Lambda\dot t$.
	The following theorem can be deduced as  a particular case of  \cite[Theorem 5.5]{CaJaSa14} and can be interpreted as a version of the Fermat's principle for a spacetime  with a causal Killing vector field when the spacetime splits globally as $S\times\R$ and the causal Killing vector field coincides with $\partial_t$, $t$ the natural coordinate on the factor $\R$.  
	\begin{theorem}[\cite{CaJaSa14}, Theorem 5.5]\label{fermat}
		Let $\sigma$ be a piecewise smooth admissible curve in $(S,F)$. Then $\sigma$ is a  geodesic of the Randers-Kropina space $(S,F)$ parametrized with $F(\dot\sigma)=1$ 
		if and only if the curve $\gamma$ defined by
		$\gamma(t)=\big(\sigma(t), t)$ is a  future-pointing lightlike pregeodesic of $(S\times \R,g)$ with non-constant component $\sigma$.  
	\end{theorem}
	\bere\label{czero}
	Let us  recall that in \cite[Theorem 5.5]{CaJaSa14} the notion of a unit geodesic of a wind Riemannian structure $\Sigma$ includes the possibility that $\sigma$ is not admissible, according to the definition given above, being a constant curve of constant value   $x_0\in S\setminus S_0$, i.e. $\Lambda(x_0)=0$.   In this case  the curve $\gamma(t)=(x_0,t)$ is 
	a lightlike pregeodesic     if and only if  $\de \Lambda_{x_0}(\mathcal D_{x_0})=0$   (see \cite[Lemma 3.21]{CaJaSa14}).
	\ere
	\section{The perturbation with standard stationary spacetimes}
	The function   $\Lambda$ in \eqref{g} is non-negative, thus adding a  positive parameter $\eps$ gives $\Lambda(x)+\eps>0$, for all $x\in S$ and then the metric  $g_\eps$	on $S\times\R$
	\[g_\eps:= g_0+\omega\otimes\de t+\de t\otimes\omega -(\Lambda +\eps)\de t^2,\]
	has larger future-causal cones than $g$. In particular, the vector field $\partial_t$ becomes timelike for $g_\eps$, remaining a Killing vector field; thus, for each $\eps>0$,
	$(S\times\R, g_\eps)$ is  a {\em standard stationary spacetime} (see, e.g.,  \cite{CaJaMa11, CaJaSa11}).
	
	Let us consider the space $\Omega_{p_0p_1}$ of continuous piecewise smooth curves $\gamma(s)=(\sigma(s),t(s))$ in $S\times \R$ parametrized on a given closed bounded interval in $\R$, say $[0,1]$, and connecting $p_0$ to $p_1$, $p_0,p_1\in S\times\R$.   By computing the Euler-Lagrange equation of the energy functional $I_\eps:\Omega_{p_0p_1}\to\R$ of the metric $g_\eps$, 
	\[
	I_\epsilon (\gamma) := \frac 1 2\int_0^1
	\left(g_0(\dot \sigma, \dot \sigma) + 2\omega(\dot \sigma)\dot t - (  \Lambda +\eps) \dot t^2\right)\de s,
	\]
	we see that the  geodesic  equation of $(S\times\R, g_\eps)$ are the following:
	\beq\label{geoeps}
	\begin{cases}
		\omega(\dot \sigma)-(\Lambda+\eps) \dot t=C_{\gamma,\eps}\\
		\nabla_{\dot \sigma}\dot \sigma=\dot t\Omega^\sharp(\dot \sigma)-\omega^\sharp\ddot t-\frac 1 2\nabla\Lambda\dot t^2
	\end{cases}
	\eeq
	where  $\omega^\sharp$ and $\Omega^\sharp$ are  the vector field and the $(1,1)$-tensor field $g_0$-metrically equivalent respectively to $\omega$ and $\Omega=\de \omega$.  
	Notice that the constant $C_{\gamma,\eps}$ is equal to $g_\eps(\dot\gamma,\partial_t)$ giving the conservation law related to the Killing vector field $\partial_t$ and the metric $g_\eps$.
	
	For $\eps=0$ the above equations reduce to the geodesic equations of $(S\times\R, g)$ and $C_{\gamma,0}=C_\gamma=g(\dot\gamma,\partial_t)=\omega(\dot\sigma)-\Lambda \dot t$.
	
	Theorem~\ref{fermat} holds also in the spacetimes $(S\times\R, g_\eps)$, but for $\eps>0$ the metric $F_\eps$ in \eqref{randerskropina} is a genuine Randers metric on $S$ which is also equal  to    
	\[F_\eps(v)=\frac{1}{\Lambda+\eps}\left(\sqrt{\big(\Lambda+\eps\big) g_{0}(v,v)+\omega^2(v)}+\omega(v)\right),\]
	and  
	\[ 
	h_\eps:=(\Lambda+\eps)g_0+\omega\otimes\omega\]
	is a Riemannian metric on $S$ for all $\eps\in(0,1]$.
	The  family of metrics $F_\eps$  is strictly  decreasing in  the parameter $\eps$ as a computation of its derivative w.r.t.  $\eps$ shows; anyway, let us  give a geometric explanation of this monotonicity.
	\bpr\label{monotonicity}
	For any $\eps_1,\eps_2> 0$, with $\eps_1<\eps_2$, and for all $v\in TS$, it holds $F_{\eps_2}(v)<F_{\eps_1}(v)$. Moreover, the same holds for $\eps_1=0$ provided that $v\in \mathcal A$, i.e. $F_\eps(v)<F(v)$, for all $\eps>0$ and $v\in\mathcal A$.
	\epr 	
	\begin{proof}                                                                                                                                 
		Let $x\in S$; by homogeneity, it is enough to prove the inequality for all vectors $v\in T_xS$ such that $F_{\eps_1}(v)=1$. Let then $v\in T_xS$ be such that $F_{\eps_1}(v)=1$. Hence, the vector $(v,1)$ is future-pointing and lightlike for the metric $g_{\eps_1}$. As $g_{\eps_1}\big((v,1),(v,1)\big)>g_{\eps_2}\big((v,1),(v,1)\big)$, we have that $(v,1)$ is future-pointing and timelike for $g_{\eps_2}$. Since  future-pointing timelike vectors $(v,\tau)$ that project on the same vectors $v\in T_x S$, $v\neq 0$, have components $\tau$ which are strictly greater than the component of the lightlike vector $\big( v, F_{\eps_2}(v)\big)$, we get $1>F_{\eps_2}(v)$, as we wanted to show. The same reasoning applies when $\eps_1=0$, provided that $v\in \mathcal A$, giving the inequality for $F_0\equiv F$ and $F_\eps$, $\eps>0$.
	\end{proof}
	Our aim is to obtain a multiplicity result for geodesics of the Randers-Kropina metric $F$,  connecting two given points, as limits of geodesics of the approximating Randers metric $F_\eps$.
	Hence, we need to study   convergence of these geodesics as $\eps\to 0$. Due to the singular nature of the Randers-Kropina metric $F$, it is convenient, thanks to Theorem~\ref{fermat}, to look at convergence of future-pointing lightlike pregeodesics  of $(S\times\R,g_\eps)$ to a pregeodesic  of $(S\times\R, g)$. 
	
	\bere\label{smoothdependence}
	Notice that the  geodesic  equation  \eqref{geoeps}, when put in normal form, are not continuous for  $\eps=0$ and $x\in S\setminus S_0$. This is an apparent singularity that disappears if we write the   geodesic  equation using the  Levi-Civita connection of the metric $g_\eps$. 
	In fact by taking, at any point $(x_0,t_0)\in S\times \R$,  a basis of $T_{(x_0,t_0)}(S\times\R)$ of the type $\{\partial_t, e_1,\ldots, e_m\}$, where $\{e_1,\ldots, e_m\}$ is a system of orthonormal vectors, w.r.t. $g_0$, of $T_{x_0}S$, one can easily  see that  $\det(g_\eps)_{(x_0,t_0)}=-(\Lambda(x_0)+\eps+\|\omega^\sharp\|^2_{x_0})$, which is not zero for all $\eps\in[0,1]$ by \eqref{lorentzian}. The computation of the Christoffel symbols of the Levi-Civita connections of the metrics $g_\eps$, $\eps\in[0,1]$ (recall that $g_0=g$),  can be performed using this basis and the expected smooth dependence on $\eps\in[0,1]$  becomes evident (see, e.g. \cite[Theorem 1]{FloSan02}; even though $\partial_t$ in this reference is timelike, the computations extend formally to the case when $\partial_t$ is somewhere lightlike).
	
	The smooth dependence of the  geodesic   equations of the metrics $g_\eps$, $\eps\in[0,1]$, on  the parameter $\eps$ would allow us to get  the existence of a geodesic of $g$ as a $C^1$-limit of a sequence of geodesics of the approximating metrics provided that their initial vectors are bounded.  In particular this applies to future-pointing lightlike geodesics.  Anyway, from Theorem~\ref{fermat}, geodesics of the approximating Randers metric $F_\eps$ lift to lightlike pregeodesics of $g_\eps$ and the reparametrizations needed to pass to  future-pointing lightlike geodesics   can produce unbounded initial vectors if $\Lambda$ vanishes at the initial point. Thus, we look for a convergence result that involves directly  pregeodesics parametrized with the temporal function $t$. This can be achieved thanks to Lemma~\ref{convexforall} plus  some basic causality properties. Our main   references for causality of spacetimes are \cite{BeEhEa96, O'neill}. 
	\ere
	We begin by giving a result about existence of a convex neighborhood for  a family of  $C^2$   pseudo-Riemannian metrics $\{g_\eps\}$, $\eps\in (-a,a)\subset\R$,  $C^2$-depending  on $\varepsilon$.  Even though this result seems natural and not surprising, we have not been able to find a reference about it. Our proof is inspired by the one about the existence of convex neighborhoods for a single metric in \cite[Proposition 5.7]{O'neill}.  
	 
	 \bere \label{family} It is convenient to see the family $\{g_\eps\}_{\eps\in (-a,a)}$ as a section on the pull-back vector bundle $\pi_2^*T^*M\otimes \pi^*_2T^*M$ over the manifold $(-a,a)\times M$ through the projection map $\pi_2:(-a,a)\times M\to M$, $\pi_2(\eps,p)=p$. By saying that the family $\{g_\eps\}_{\eps\in (-a,a)}$ is $C^2$ we mean that this section is $C^2$. 
	 \ere

	\bl\label{convexforall}
	Let $\{g_{\varepsilon}\}_{  \varepsilon\in (-a,a) }$  be a  $C^2$ family (in the sense of Remark~\ref{family})  pseudo-Riemannian metrics on a manifold $M$.    For every  point $p_0\in M$ and every $\varepsilon_0\in(0,a)$ there exists  a neighborhood of $p_0$ in $M$  which is simultaneously convex for all  metrics $g_\varepsilon$,  $\varepsilon\in [-\varepsilon_0,\varepsilon_0]$. 
	\el
	\begin{proof}
		Let $m$ be the dimension of $M$, $\pi:TM\to M$ be the canonical projection and for any $p\in M$ let $\exp^\varepsilon_p$ be  the exponential map of $g_{\varepsilon}$ at $p\in M$.   Let $\bar \eps\in [-\eps_0,\eps_0]$.  By standard results about smooth dependence of solutions of differential equations on  initial conditions and parameters, taking into account the homogeneity properties of the  geodesic   equations, we know that there exists a neighborhood  $ \mathcal O_{\bar\eps} \subset TM$  of $(p_0,0)$ and   $\delta_{\bar\eps}>0$ such that  the map $\varphi:\mathcal O_{\bar\eps}\times (\bar \eps-\delta_{\bar\eps}, \bar\eps+\delta_{\bar\eps})\rightarrow M \times M \times (\bar \eps-\delta_{\bar\eps}, \bar\eps+\delta_{\bar\eps})$,  
		\[\varphi(v,\varepsilon)=(\pi(v),exp^\varepsilon_{  \pi(v) }(v),\varepsilon)\]
		is well-defined   and  $C^1$.    Since  $d\varphi_{(p_0,0, \bar\eps )}$ is  the identity map,   by the inverse function theorem, we get  that  there exists an open neighborhood $ U_{\bar\eps} \subset M$ of $p_0$ and $ \beta_{\bar\eps}\in(0,\delta_{\bar\eps}) $ such that, if $ \mathcal  U_{\bar\eps}=: \varphi^{-1}\big (U_{\bar\eps}\times U_{\bar\eps}\times  (\bar \eps-\beta_{\bar\eps}, \bar\eps+\beta_{\bar\eps}) \big)$,  then $\varphi|_{ \mathcal U_{\bar\eps} }$ is a   $C^1$-diffeomorphism. Being the interval $[-\eps_0,\eps_0]$ compact, we obtain a neighborhood $U$ of $p_0$ as intersection of a finite number of neighborhoods of the type  $U_{\bar\eps}$ above  and $\mathcal  U=: \varphi^{-1}\big ( U\times U\times  [-\varepsilon_0,\varepsilon_0] \big)$, such that  the map $\varphi|_{\mathcal U}$ still is a $C^1$-diffeomorphism.   If we want to prove that it is possible to choose a neighborhood of $p_0$, $U'\subset U$, convex for all $g_\varepsilon$ with $\varepsilon\in  [-\varepsilon_0,\varepsilon_0] $, we need to prove that given two points $q_0,q_1\in U'$ and $v=\varphi^{-1}(q_0,q_1,\varepsilon)\in T_{q_0}M$, the $g_\varepsilon$-geodesic $\sigma^\varepsilon_v:[0,1]\rightarrow M$ is contained in $U'$.  
		By choosing $U$ smaller if necessary, we can assume that it is contained in the domain $V$ of a chart $\psi=(x^1,\ldots,x^m)$ around $p_0$ with $x^1(p_0)=\ldots=x^m(p_0)=0$.

		Define the function $N:V\rightarrow \R$, given by 
		\beq\label{defN} N(p)=x^1(p)^2+\ldots+x^m(p)^2\eeq 
		and the $(0,2)$ tensor $B^\eps$ on $V$ with coordinates 
		\beq\label{B}B^\eps_{ij}(p)=\delta_{ij}-\sum_{k=1}^mx^k(p)\Gamma_{ij}^k(p,\varepsilon),\eeq
		being $\Gamma_{ij}^k(p,\varepsilon)$ the Christoffel symbols of $g_\varepsilon$. Let us define
		\[\Vn(\delta)=\{p\in V: N(p)<\delta\}.\] 
		As $B^\eps$ is positive definite at $p_0$, by taking $\delta>0$ small enough (recall that $\Gamma^k_{ij}(p,\varepsilon)$  are $C^1$ ), we can assume that $\overline{\Vn(\delta)}$ is contained in $U$ and $B^\eps$ is positive definite in $\overline{\Vn(\delta)}$, for all $\varepsilon\in[-\varepsilon_0,\varepsilon_0]$. Let us show that we can choose $U'=\Vn(\delta)$ as a convex neighborhood of $p_0$ for all $g_\varepsilon$. Assume on the contrary that $\sigma^\varepsilon_v$ defined as above from $q_0$ to $q_1$, both in $\Vn(\delta)$, is not completely contained in $\Vn(\delta)$. Let us show that in this case, there has to be,    for each $\eps\in [-\eps_0,\eps_0] $,   a geodesic contained in $\overline{\Vn(\delta)}$ and tangent to the boundary of $\Vn(\delta)$   at one of its inner points.   Indeed, as $\Vn(\delta)$ is connected, there exists a smooth curve $\alpha:[0,1]\rightarrow \Vn(\delta)$ from $q_1$ to $q_0$. Define,    for each $\eps\in [-\eps_0,\eps_0] $,    $\beta^\eps=(\exp^\varepsilon_{q_0})^{-1}\circ\alpha$ (recall that $\exp^\varepsilon_{q_0}$ is invertible on $U$ which contains $\Vn(\delta)$). The curve $\beta^\eps$ provides a  $C^1$   variation $\{\gamma^\eps_t\}_{t\in[0,1]}$ of $g_\varepsilon$-geodesics departing from $q_0$ with $\dot\gamma^\eps_t(0)=\beta^\eps(t)$. For $t=1$, we obtain a constant curve, and then for $t$ close to $1$, these geodesics are contained in $\Vn(\delta)$. Let $ t_\eps \in(0,1)$ be such that $\gamma^\eps_{ t_\eps }$ is the first geodesic completely contained in $\overline{\Vn(\delta)}$   (recall that all the geodesics $\gamma_t^\eps$ must return in $\Vn(\delta)$ as their final point is on $\alpha$ which is contained in $\Vn(\delta)$).   By the  differentiability   of the variation, this geodesic must be tangent to the boundary as required. Let us show that this gives a contradiction. Let $s_0\in(0,1)$ be the instant where $\gamma^\eps_{ t_\eps }$ touches the boundary of $\Vn(\delta)$. Then the function $s\in[0,1]\mapsto N(\gamma^\eps_{ t_\eps }(s))$  (recall \eqref{defN}) has a maximum at $s_0$. Moreover,
		\[\frac{d^2}{ds^2}(N\circ \gamma^\eps_{ t_\eps })=2\sum_{i=1}^m(\left(\frac{d\gamma^{\eps\, i}_{ t_\eps }}{ds}\right)^2+\gamma^{\eps\, i}_{ t_\eps }\frac{d^2\gamma^{\eps\, i}_{ t_\eps }}{ds^2})\]
		denoting $\gamma^{\eps\, i}_{ t_\eps }=x^i\circ \gamma^{\eps}_{ t_\eps }$, and using the  geodesic  equation 
		\begin{align*}
			\frac{d^2}{ds^2}(N\circ \gamma^\eps_{ t_\eps })=&2\sum_{i=1}^m(\left(\frac{d\gamma^{\eps\, i}_{ t_\eps }}{ds}\right)^2
			-\sum_{j,k=1}^m \gamma^{\eps\,  i}_{ t_\eps }\Gamma_{jk}^i(\gamma^\eps_{ t_\eps }(s),\eps)\frac{d\gamma^{\eps\, j}_{ t_\eps }}{ds}\frac{d\gamma^{\eps\, k}_{ t_\eps }}{ds}).
		\end{align*}
		Therefore,
		\begin{align*}
			\frac{d^2}{ds^2}(N\circ \gamma^\eps_{ t_\eps })(s_0)=& 2B^\eps(\dot\gamma^\eps_{ t_\eps }(  s_0 ),\dot\gamma^\eps_{ t_\eps }(s_0))>0,
		\end{align*}
		which is a contradiction with the fact that $N\circ \gamma^\eps_{ t_\eps }$ must have a maximum at $s_0\in (0,1)$.  This concludes that $\Vn(\delta)$ is a normal neighborhood of all its points and then a convex neighborhood for all $g_\eps$,  $\varepsilon\in [-\varepsilon_0,\varepsilon_0] $.
	\end{proof}
	\begin{remark}
		  Let us observe that the proof of Lemma \ref{convexforall} works in the more general case of a  $C^1$  family of homogeneous anisotropic connections \cite{Jav19,Jav20} word for word. In coordinates, an anisotropic connection is equivalent to have Christoffel symbols which depend on oriented directions, namely, they are positive homogeneous functions of degree zero on the slit tangent bundle of the chart. In this case, the autoparallel curves of the anisotropic connection satisfy the same formal equations of geodesics and solutions are preserved by affine reparametrizations.  More generally,  a slight modification of the proof of Lemma~\ref{convexforall} gives also  a common convex neighborhood for  a  $C^1$  family of sprays. Let us recall that a {\em spray} on a smooth manifold $M$ (see, e.g., \cite{CrMeSa12, Jav19}) is  a vector field $S$ on $TM$ which, in the domain  $U\times\R^m$  of any natural chart of $TM$, is given by 
		  \[
		  	S(v)=y^i \frac{\partial}{\partial x^i}- 2G^i(v)\frac{\partial}{\partial y^i},
		  \]
		  where the coordinates of $v\in TM$ are $(x^1,\ldots,x^m, y^1,\ldots,y^m)$, and $G^i:TU\rightarrow \R$ are positively homogeneous of degree $2$ in $v$, i.e. $G^i(\lambda v)=\lambda^2 G^i(v)$.  The  autoparallel curves of a spray are the projections on $M$ of the integral curves of $S$ in $TM$.  In a system of coordinates  $\big (U, (x^1,\ldots,x^m)\big)$ of $M$,  a curve $\gamma:[a,b]\rightarrow U\subset M$ is then  an autoparallel curve of the spray $S$ if and only if
		  \beq\label{spray}
		  	\ddot \gamma^i=- 2G^i(\dot \gamma).
		  \eeq
		  It is well-known that   a  spray carries with it  a well-defined exponential map (this follows from \eqref{spray} and the homogeneity property of $G^i$). We notice here that the minimal requirement to repeat the reasoning in the proof of Lemma~\ref{convexforall} is that a family  of sprays, depending on a parameter $\eps\in (-a,a)\subset \R$,  has coefficients $G^i(v,\eps)$ which are $C^1$ functions on $U\times \R^m\times(-a,a)$. Although  in this case  the  correspondence between solutions of \eqref{spray} and autoparallel curves of a uniquely determined Berwald connection (see \cite[\S 4.7]{BucMir07}, \cite[\S 4.3]{Jav19})  does not make sense (due to the lack of the second partial derivatives w.r.t. the components of $v$ of the coefficients $G^i$),  we can replace, for each $\eps\in(-a,a)$, the bilinear tensor $B^\eps$ in \eqref{B} by the function
$\mathcal B^\eps: V\times \R^m\to \R$ defined as 
\[\mathcal B^\eps(x,v)= |v|^2 -\sum_{i=1}^{m} x^i G^i(v,\eps),\]
(compare also with \cite[Appendix B]{CrMeSa12}, where the existence of a convex neighborhood for a single spray is proved). 
The second derivative of  $N\circ \gamma^\eps_{t_\eps}$ is then  equal to 
$2\mathcal B^\eps\big (\gamma^\eps_{t_\eps}(s),\dot\gamma^\eps_{t_\eps}(s)\big)$ which is positive in a small neighborhood of $x=0$, uniformly w.r.t $v\in \R^m\setminus\{0\}$ and $\eps\in [-\eps_0,\eps_0]$, since by the homogeneity of $G^i$ we have 
\[\mathcal B^\eps(x,v)= |v|^2 \big (1 -\sum_{i=1}^{m} x^i G^i(v/|v|, \eps) \big).\]
 In particular, we  obtain a common convex neighborhood for   families of pseudo-Finsler metrics defined on $TM$ which are at least $C^3$ on $TM\setminus 0$ (see e.g. \cite[\S 4.1]{Jav19} for the definition of a pseudo-Finsler metric) because their geodesics can be obtained as autoparallel curves of  $C^1$ sprays (see e.g. \cite[Theorem 5.4.1]{BucMir07}).   Finally, Lemma \ref{convexforall} can be easily generalized  to    families depending on a finite number of parameters.  
	\end{remark}
		Our aim is now to state a convergence result for a sequence $\{\gamma_n\}$ such that,  for each $n\in\N$,  $\gamma_n$ is a  future-pointing lightlike geodesics of $g_{\eps_n}$, with $\eps_n\to0$.   
	
	Let us first recall  some basic causality notions. Let $(M,g)$ be any  spacetime, i.e. a connected, time-oriented Lorentzian manifold; for any point $p\in M$, the {\em chronological future of $p$}, 
	$I^+(p)$ (resp. the {\em causal future of $p$}, $J^+(p)$) is  the set of points in $M$ that can be reached from $p$ by a future-pointing timelike (resp. causal) curve. We recall that $I^+(p)$ is an open set and $p\in J^+(p)$. We recall that a continuous curve $\gamma:I\to M$, $I$  an interval in $\R$,  is said future-pointing and causal if  for each $s_0\in I$ there exists  a convex neighborhood $U$ of $\gamma(s_0)$ and $\epsilon>0$ such that $\gamma\big((s_0-\epsilon, s_0+\epsilon)\big)\subset U$ and for all $s_1, s_2\in (s_0-\epsilon, s_0+\epsilon)$ with $s_1<s_2$, there exists a smooth future-pointing causal curve in $U$ from $\gamma(s_1)$ to  $\gamma(s_2)$ (see \cite[p.54]{BeEhEa96}).
	For $p\in S\times\R$, let us denote by $I^+_\eps(p)$ and $J^+_\eps(p)$ the chronological and the causal future of $p$ in $(S\times\R, g_\eps)$.

	Moreover, let us denote by $E_\eps$ the energy functional of the Randers metric $F_\eps$, i.e. for a given  couple of points $x_0,x_1\in S$,  $E_\eps:\Omega_{x_0x_1}\to \R$, $E_\eps(\sigma)=\frac1 2\inte F_\eps^2(\dot\sigma)\de s$,  moreover, 
	for each $\eps>0$, let $\ell_\eps:\Omega_{x_0x_1}(S)\to [0,+\infty)$ be the length functional of the Randers metric $F_\eps$, i.e. $\ell_\eps(\sigma)=\int_{\sigma}F_\eps$.

	Let us finally state  a fundamental   assumption to get the convergence result and our main  result, Theorem~\ref{multiplegeos}.  To this end, let us introduce also some basic elements of  Finsler  geometry. For any couple of points $x_0,\ x_1\in S$ we can define 
	$d_\eps(x_0,x_1):=\inf_{\sigma\in \Omega_{x_0x_1}}\ell_\eps(\sigma)$, which is a non-symmetric distance on $S$.  Then for any $x_0\in S$ and $r>0$, let $B_\eps^+(x_0,r):=\{x\in S:d_\eps(x_0, x)<r\}$ and  $B_\eps^-(x_0,r):=\{x\in S:d_\eps(x, x_0)<r\}$ be, respectively, the forward and the backward ball centred  at  $x_0$, having radius $r$ (see, e.g. \cite{BaChSh00}). We denote their closures by $\bar B_\eps^{\pm}(x_0,r)$.
	
	Let us consider the following condition:
	\begin{align}\begin{split}&\text{there exists $\bar \eps\in (0,1]$ such that  the intersections of the}\nonumber \\ 
			&\text{closed forward and the backward balls of $F_{\bar\eps}$ are compact.}\end{split}\tag{$\star$}\label{star}\end{align}
	\bere\label{aboutstar}
	Condition \eqref{star} is equivalent to global hyperbolicity of $(S\times\R,g_{\bar \eps})$ by \cite[Proposition 2.2 and Theorem 4.3]{CaJaSa11}. Moreover, it  implies that the energy functional of $F_{\bar\eps}$ (and thus of each $F_\eps$ with $0<\eps<\bar\eps$) satisfies the Palais-Smale condition on   $W^{1,2}_{x_0x_1}(S)$   
	(see the comments above \cite[Theorem 5.2]{CaJaSa11}). 
	\ere 
	\bere Let $(S, g_0)$ be a complete Riemannian manifold and let $d_0$ be the distance associated with $g_0$.
	By \cite[Proposition 3.1 and Corollary 3.4]{Sanche97a} we know that  if there exist a point $\bar x$ and a positive constant $C$  such that  $\|\omega\|_x\leq C(d_0(x,\bar x)+1)$, for all $x\in S$, and there exists a constant $L\geq 0 $ such that $0\leq \Lambda(x)\leq L$ for all $x\in S$ then the standard stationary spacetimes $(S\times\R, g_\eps)$ are globally hyperbolic, for each $\eps>0$ and therefore condition \eqref{star} holds.
	\ere
	 We point out that the convergence result below can be obtained using Lemma~\ref{convexforall} and \cite[Lemma 5.7]{CaJaSa14} about the existence of a limit curve of a sequence of future-pointing, future inextendible, causal curves in a spacetime endowed with a global temporal function. Anyway, we give here a direct proof for the sake of reader convenience.  
	\bpr\label{preconv}
	Assume that condition \eqref{star} holds   and let $\gamma_n:[t_0, t_0+  \Delta_n]\rightarrow S\times\R$ be a sequence of curves of $S\times\R$ parametrized by the time-coordinate $t$, $\gamma_n(t)=\big(\sigma_n(t),t\big)$,  in such a way that $\gamma_n(t_0)=(x_0,t_0)$ and $\gamma_n(t_0+\Delta_n)=(x_1,t_0+\Delta_n)$ for all $n\in\N$, with $x_0,x_1\in S$ and $t_0\in\R$ and such that $\gamma_n$ is a lightlike pregeodesic of $g_{\varepsilon_n}$,  $\Delta_n\to \Delta\in\R$ and $\eps_n\to 0$.   Then there exists a subsequence of $\{\gamma_n\}$ that uniformly converges to a future-pointing lightlike  pregeodesic  of the metric $g$  from $(x_0,t_0)$ to $(x_1,t_0+ \Delta)$.  
	\epr
	\begin{proof}
		From Theorem~\ref{fermat} applied to each spacetime $(S\times\R, g_{\eps_n})$ and each Fermat metric $F_{\eps_n}$ (recall that, for all $\eps>0$,  $F_\eps$ is a Randers metric and then it is a Randers-Kropina metric too) and  Proposition~\ref{monotonicity} we have, for all $n$ big enough  and all   $t\in [t_0,t_0+\Delta_n]$, 
		\bmln d_{\bar\eps}\big(x_0, \sigma_n(t))\big)\leq \int_{ t_0}^tF_{\bar\eps}\big(\dot\sigma_{n}(r)\big)\de r< \int_{ t_0}^tF_{\eps_n}\big(\dot\sigma_{n}(r)\big)\de r\\\leq \int_{t_0}^{ t_0+\Delta_n}F_{\eps_n}\big(\dot\sigma_{n}(r)\big)\de r= \Delta_n\emln
		and analogously $d_{\bar\eps}\big(\sigma_{n}(t), x_1\big)< \Delta_n$. Hence, the supports of the curves $\gamma_n$ are  contained in $\big(\bar B_{\bar\eps}^+(x_0, \Delta+1)\cap \bar B^-_{\bar\eps}(x_1,\Delta +1)\big)\times[t_0,t_0 +\Delta+1]$, for all $n$ big enough,  which is compact in $S\times\R$ by condition \eqref{star}. Therefore, there exists a positive constant $K$ depending on this compact set  such that, 
		\[ \sqrt{g_0\big(\dot\sigma_n(t),\dot\sigma_n(t)}\big)\leq K F_{\bar\eps}\big(\dot\sigma_n(t)\big).\]
		For each $n$, we extend  $\gamma_n$ to a future inextendible causal curves of $g_{\varepsilon_n}$ with the vertical lines from $(x_1, t_0+\Delta_n)$.
		Let us denote with $d_0$ the distance associated  with  $g_0$. For all $\eps\in (0,\bar\eps]$ and for all $t_0\leq t_1\leq t_2\leq  t_0+\Delta+1$ we get, taking into account that if $t>t_0 +\Delta_n$ the component on $S$ of the extended curve is constant,
		\[ d_0\big(\sigma_n(t_1), \sigma_n(t_2))\big)\leq K\int_{t_1}^{t_2}F_{\bar\eps}\big(\dot\sigma_{n}(r)\big)\de r\\ \leq  K \int_{t_1}^{t_2}F_{\eps_n}\big(\dot\sigma_{n}(r)\big)\de r\leq  K (t_2-t_1),\]
		for all $n$ big enough,  where we have used Prop. \ref{monotonicity}. 
		Hence, for any  $t_1,\ t_2\in [t_0,t_0+\Delta+1]$ we have 
		\[d_0\big(\sigma_n(t_1), \sigma_n(t_2)\big)\leq K  |t_2-t_1|,\] 
		for all $n$ big enough. Therefore, the sequence $\{\sigma_n\}$ is equi-Lipschitz and equi-bounded and by Ascoli-Arzel\`a theorem, there exists a subsequence, still denoted by $\{\sigma_n\}$ which uniformly converges on $[t_0,t_0+\Delta+1]$ to a Lipschitz curve $\sigma$. Notice that being $\sigma_n(t_0+\Delta_n)=x_1$, and $\Delta_n\to \Delta$, we also have
		\[d_0(\sigma(t_0+\Delta), x_1)\leq d_0\big (\sigma(t_0+\Delta),\sigma(t_0+\Delta_n)\big)+d_0\big (\sigma(t_0+\Delta_n),\sigma_n(t_0+\Delta_n)\big)\to 0,\] 
		hence $\sigma(t_0+\Delta)=x_1$. Then  $\gamma_n$ uniformly converges (w.r.t. to the metric induced by $g_0\oplus\de t^2$ on $S\times\R$) to the  curve $\gamma(t)=(\sigma(t), t)$ and $\gamma( t_0+\Delta)=(x_1, t_0+\Delta)$. Notice that being $\gamma$  the uniform limit of curves  which are causal and future-pointing in $(S\times\R,g_{\eps_n})$, for each $\eps_n\in  (0,\bar\eps) $, then it is causal and future-pointing in $(S\times\R,g_{\eps_n})$ for each $\eps_n\in  (0,\bar\eps) $. In fact, it is enough to take, for each point $\gamma(t)$, a convex neighborhood of all the metrics $g_{\eps_n}$ (recall Lemma~\ref{convexforall})  and then the claimed properties of $\gamma$ follow as in the proof of \cite[Lemma~3.29]{BeEhEa96}. Observe that since $\gamma$ is causal and future-pointing for all the $g_{\varepsilon_n}$, $n$ big enough,  its  velocity  is contained in the future causal cones of $g_{\varepsilon_n}$ for almost all the instants and all $n$ big enough and then $\gamma$ is also causal and future-pointing for $g$ (see \cite[Theorem A.1]{CaFlSa2008}).
		Let us show that $\gamma$ is the required lightlike pregeodesic of $g$. By contradiction, let us assume that there exists $t_0\in [t_0,t_0+\Delta)$ and an interval $I$ containing $t_0$ and which is strictly contained in  $[t_0,t_0+\Delta)$, such that $\gamma|_I$ is not a lightlike pregeodesic. By Lemma~\ref{convexforall}, we can consider    a convex neighborhood $U$ of $\gamma(t_0)$ for all the metrics $g_{\eps_n}$ and,  without loosing generality, we can assume that there is some $\tau>0$, such that   $[t_0-\tau,t_0+\tau]= I$  and $\gamma(t_0\pm\tau)\in U$ (if $t_0=0$ it is enough to consider the interval $[0,\tau]$, i.e. the role of $t_0-\tau$  is played by $0$). Since we are saying that $\gamma|_I$ is not a lightlike pregeodesic, being a causal curve there must exist a timelike geodesic of the metric $g$ between $\gamma(t_0-\tau)$ and $\gamma(t_0+\tau)$ and contained in $U$. As $\gamma_n$ converges uniformly to $\gamma$ on $I$, we have that for $n$ big enough $\gamma_n(I)\subset U$ and being  the chronological relation  open,  there is   also  a timelike geodesic (for the metric $g_{\eps_n}$)  between $\gamma_n(t_0-\tau)$ and $\gamma_n(t_0+\tau)$ contained in $U$, in contradiction with the fact that  $\gamma_n|_I$ is a lightlike pregeodesic in $U$ and, being $U$ a convex neighborhood of $g_{\eps_n}$,  it is then the only pregeodesic  in $U$  between $\gamma_n(t_0-\tau)$ and $\gamma_n(t_0+\tau)$.	\end{proof}

	From Proposition~\ref{preconv} then a  result about convergence of geodesics of the approximating Randers metrics $F_\eps$ to a geodesic of $F$ follows. 
	\begin{proposition}\label{energies}
		Assume that condition \eqref{star}  holds;  let    $x_0\neq x_1\in S$ and, 
		for each $\eps\in(0,1]$,    let $\sigma_\eps:[0,1]\to S$  be a geodesic of the Randers metric $F_\eps$  such that $\sigma_\eps(0)=x_0$ and $\sigma_{  \varepsilon }(1)=x_1$  and let us assume that   
		$\sup_{\eps\in (0,1]} E_\eps(\sigma_\eps)=E<+\infty$.  
		
		Then there exists a sequence $\eps_n\to 0$ and a geodesic $\sigma$ of $(S,F)$ from  $x_0$ to $x_1$ such that   $\sigma_{\eps_n}$ uniformly converges to $\sigma$ and $\ell_{\eps_n}(\sigma_{\eps_n})\to  \ell_F(\sigma)$.  
		
		Moreover,   if $\Lambda(x_0)\neq 0$ or $\Lambda(x_0)=0$ but $\de\Lambda_{x_0}(\mathcal D_{x_0})\neq 0$,   then the point $x_1$ can be taken to be equal to $x_0$.
	\end{proposition}	
	\begin{proof}
		    As any Fermat  metric $F_\eps$ is also a Randers-Kropina metric, we know from Theorem~\ref{fermat} that $\gamma_\eps:[0, \sqrt{2E_\eps(\sigma_\eps)} ]\to S\times \R$,   defined as \[\gamma_\varepsilon(t)=\big(\sigma_\varepsilon( t /\sqrt{2E_\eps(\sigma_\eps)}) ,t\big),\]   
		     is a   future-pointing lightlike pregeodesic of $(S\times\R,g_\eps)$.  From    Proposition~\ref{preconv} we then get  the existence of a sequence $\{\eps_n\}$ such that $\gamma_{\eps_n}$ uniformly converges to future-pointing lightlike pregeodesic $\gamma=(  \tilde\sigma (t),t)$ of $g$ between $(x_0, 0)$ and $(x_1, \sqrt{2\mathcal E})$,   where $\mathcal E\in [0,E]$ is such that $\lim_n E_{\eps_n}(\sigma_{\eps_n})=\mathcal E$.    Taking into account that being $x_0\neq x_1$,  $\tilde\sigma$ cannot be constant,   from Theorem~\ref{fermat} we deduce that $\tilde\sigma$ is a unit  geodesic of $F$  and $\sqrt{2\mathcal E}=\ell_F(\tilde\sigma)$.     The reparametrization $\sigma(t)=\tilde\sigma( t \sqrt{2 \mathcal E })$ is the required solution.    For the last part, if $x_0=x_1$, again $\tilde\sigma$ cannot be constant  by Remark~\ref{czero} and then Theorem~\ref{fermat} can be applied also in this case. 
	\end{proof}
	\section{The \bb control system \eb and admissible curves}
	We introduce in this section \bb a control system \eb whose solutions are    $W^{1,p}$-admissible curves $\sigma :[0,1]\to S$.   
	
	 Let $p\in \R$, $p\geq 1$ and let  $W_{x_0x_1}^{1,p}(S)$ be the Sobolev manifold of $W^{1,p}$ paths between two points $x_0$ and $x_1$ in $S$ defined as
	\bmln W^{1,p}_{x_0x_1}(S):=\{\sigma:[0,1]\to S: \text{ $\sigma$ absolutely continuous,}\\ \sigma(0)=x_0,\ \sigma(1)=x_1,\ \inte \big(g_0(\dot\sigma,\dot\sigma)\big)^{p/2}\de s<+\infty\},\emln  
	 (of course, if $p=1$ the integrability of $\sqrt{g_0(\dot\sigma,\dot\sigma)}$ is equivalent to the absolute continuity of $\sigma$).  
	Let us assume that $\omega_x\neq 0$ for all $x\in S$ and let $\mathcal D\subset TS$ be the distribution defined by its kernel, i.e. for each $x\in S$, $\mathcal D_x:=\ker(\omega_x)$.
	Let $X_1,\ldots X_d$, $d\geq \mathrm{rank}\mathcal D$ be a set of globally defined smooth vector fields on $S$ which generate $\mathcal D$ {\color{black} and are compatible with the the sub-Riemannian norm on $\D$ defined by $g_0|_{\D\times \D}$, in the sense that, for all $v\in D$, it holds
		\beq\label{subRiem} 
		g_0(v,v)=\min\{\sum_{i=1}^{d} u_i^2\,\, \big|\,  v=\sum_{i=1}^{d}u_i (X_i)_x\},  \eeq
		where $x=\pi(v)$ (see \cite[Corollary 3.27]{AgBaBo20}).
		
		 Following \cite{CaMaSu24}, we consider the system
	\beq\label{affinecontrol}
		\dot\sigma=u_0X_{0\sigma}+\sum_{i=1}^du_i X_{i\sigma},\quad \sigma(0)=x_0\in S	\eeq
	where 	$X_0$ is the smooth $g_0$-unit vector field on $S$ which is $g_0$-orthogonal to  $\mathcal D$, such that   $-\omega(X_0)=\|\omega\|$, and by $X_{j\sigma}$, $j\in\{0,\ldots, d\}$ we denote the vector field $X_j$ along $\sigma$.
	
		Notice that for a solution $\sigma$ of \eqref{affinecontrol} with $u_0>0$, we have $\omega(\dot\sigma(s))=u_0\omega(X_0)=-u_0\|\omega\|<0$, hence  $\dot\sigma(s)\in\mathcal A$ for  a.e.   $s\in I$. 

Actually the  controls  $u:I\to \R^{d+1}$, $u=(u_0, u_1, \ldots, u_d)$ that we are going to consider satisfy some constraints: the first components  $u_0$ are non-negative piecewise constant functions,  while the other ones constitute a  specific set of $L^\infty$ vector-valued functions. 
In order to define this last set,  we need that  $S$ admits a covering that satisfies some conditions below (see \cite[Remarks 2.1 and 2.2]{CaMaSu24}).
\bere\label{localsections}
For each  $z\in M$
we  consider  a neighbourhood  $U_z$ of $z$ and 
a frame field $\{Y_i\}_{i\in \{1,\ldots, m-1\}}$ for the distribution $\D$   such that 
$g_0(Y_i, Y_i)=1$ on  $U_{z}$, for all $i\in\{1,\ldots m-1\}$. 
Being $\D$ nowhere integrable i.e. $(\omega\wedge\de \omega)_p\neq 0$, for all $p\in S$, up to 
restricting $U_z$, there exist $Y_j, Y_l$ in the basis, such that $[Y_j, Y_l]$ is transversal to $\D$ at each point in  $U_z$ and moreover   
\beq\label{lambday}
\lambda_z:= \inf\big(-\omega([Y_j,Y_l])\big)=\inf d\omega(Y_j,Y_l)>0.\eeq
By rearranging the vector fields in the local frame, we can assume that the vector field $Y_j, Y_l$ are the first two $Y_1, Y_2$.
\ere
\smallskip
We make the following assumptions:
\begin{enumerate}[(i)]
	\item\label{ass1}
	\[\Omega:=\sup\|\omega\|<+\infty;\]
	\item\label{ass2}  there exists  a  countable  covering    $U_{z_k}$, $k\in \mathbb N$,  of  $M$ made  by neighbourhoods $U_{z_h}$ as in Remark~\ref{localsections} and 
	\beq\label{lambda}
	\lambda:=\inf_{k\in \N } \lambda_{z_k}>0.
	\eeq
\end{enumerate}
\bere
We notice that assumptions~\eqref{ass1}--\eqref{ass2} are satisfied if $S$ is compact.
\ere 
Let 
us   multiply   the vectors $Y_i$, $i\in\{1,\ldots,m-1\}$,  by the same factor
\[C:=\big(5(m+3)\Omega/\lambda\big)^{1/2}.
\]
Let us denote these rescaled vector fields still with $Y_i$. Hence,  for the rescaled vector fields $Y_i$ the number defined in \eqref{lambda}, still denoted with $\lambda$, satisfies: 
\[\lambda >4(m+3)\Omega,\]
moreover
\beq g_0(Y_i, Y_i)= C^2,  \text{ on } U_{z_k},\ k\in \N, \label{k1k2}\eeq
for all $i\in\{1,\ldots m-1\}$.

\medskip
\noindent {\bf Definition of the admissible controls \A. }
Let $D$ be the set of partitions of the interval $I$.     
We call a measurable control $u\colon I\to \R^{d+1}$ {\it admissible} if there exists $P\in D$ such that for each  $J\in P$,  there exist  a constant  $\xi_J\ge 0$, and  measurable functions  $\alpha_{Ji}\colon J\to \R$, $i\in\{1,\ldots,d\}$,
with 
\bal &\sup_{s\in J} \sum_{i= 1}^d\alpha_{Ji}^2(s)\leq C^2,\label{Csquare}\\ 
\intertext{where $C$ satisfies \eqref{k1k2}, such that}
&u_0|_J=\xi_J^2,\label{u0}\\
&u_i|_J=\xi_J\alpha_{Ji}, \ \text{ for all $i\in \{1,\ldots, d\}$.}\label{u_i}
\eal
We define \A	to be the set of all admissible controls. 

Notice that  if $u=(u_0, u_1, \ldots, u_d)\in \A$, then from \eqref{u0},  $u_0$ is nonnegative and piecewise constant, and  from \eqref{u_i}, each $u_i$, $i\in\{1,\ldots, d\}$, is 
bounded, hence $\A\subset L^{\infty}(I, \R^{d+1})$. Moreover, $u|_J=0$ if and only if $\xi_J=0$.

\smallskip
Let us denote by  $\Omega^{1,p}\subset $  the set of solutions of \eqref{affinecontrol} corresponding to controls $u=(u_0,\ldots , u_d)\in \A$.
}

Let $\mathcal F: \Omega^{1,p} \to S$ be the {\em endpoint map}, $\mathcal F(\sigma)=\sigma(1)$.
	Let us denote by $ \Omega^{1,p} (x_1)$ the set $\mathcal F^{-1}(x_1)$, i.e. the set of $ W^{1,p} $-solutions of \eqref{affinecontrol} connecting $x_0$ to $x_1$ \bb and stemming from controls   $u\in \A$. \eb
	
	\bb By  Theorem 3.2  in \cite{CaMaSu24} we deduce the following: \eb
	\bt\label{lerario}
	Assume that $\mathcal D$ is \bb  nowhere integrable (i.e $(\omega\wedge\de \omega)_p\neq 0$, for all $p\in S$)  and satisfies assumptions \eqref{ass1}--\eqref{ass2}. \eb  Then the set $ \Omega^{1,p}(x_1) $ is homotopically equivalent to $ W^{1,p}_{x_0x_1}(S) $, being the inclusion map 
	\[ j_p :\Omega^{1,p}(x_1)\to W^{1,p}_{x_0x_1}(S) \] 
	a homotopy equivalence.
	\et

	Let us recall that the {\em Lusternik-Schnirelman category} of a subset $A$ of a topological space $B$, denoted by $\mathrm{cat}_B(A)$ is the minimum number of closed contractible (in $B$) subsets of $B$ that cover $A$.  
	
	\bere\label{catK}\bb Let $f\colon W^{1,p}_{x_0x_1}(S) \to \Omega^{1,p}(x_1)$ be a homotopy inverse of $j$. Let $V$ be  a subset of $W^{1,p}_{x_0x_1}(S)$ such that   
	$\mathrm{cat}_{W^{1,p}_{x_0x_1}(S)}(V)\geq k$. Then  \[\mathrm{cat}_{W^{1,p}_{x_0x_1}(S)}\big(j(f(V))\big)\geq k,\] 
	otherwise  there exist $h<k$ closed contractible sets $U_i\subset  W^{1,p}_{x_0x_1}(S) $ that cover $j(f(V))$ and then $f^{-1}\big(j^{-1}(U_i)\big)$ cover $V$ and are contractible because the identity map of 
	$W^{1,p}_{x_0x_1}(S)$ is homotopic to $j\circ f$. \eb 
	\ere
	Let us show that the energy functional $E$ of the Randers-Kropina metric is  bounded on compact subsets which are  the images through \bb $j\circ f$ of compact subsets in $W^{1,2}_{x_0x_1}(S)$.  \eb 
	\bpr\label{bounded}
	 \bb Let $K\subset W^{1,2}_{x_0x_1}(S)$ be a compact subset and let us consider the inclusion $j:\Omega^{1,2}(x_1)\to W^{1,2}_{x_0x_1}(S)$ and a homotopy inverse $f$ of $j$.  Then 
	 \[\sup_{\sigma\in j(f(K))}E(\sigma)<+\infty.\]	\eb \epr
	\begin{proof}
	{\color{black}	
		Since $j(f(K))$ is compact in $W^{1,2}_{x_0x_1}(S)$,  and this space embeds continuously in $C^0(I, S)$, we deduce that the supports of the curves  $\sigma\in j(f(K))$  are in a compact subset $K_1$ of $S$. As $\omega_x\neq 0$, for all $x\in S$,  there  exists then a positive constant $\delta=\min_{x\in K_1} \|\omega_x\|$. As the curves $\sigma\in j(f(K))$ solve \eqref{affinecontrol} we have       
		\[g_0(\dot\sigma,\dot\sigma)=u_0^2g_0(X_0, X_0)+g_0(\dot\sigma_\D, \dot\sigma_\D)=u_0^2+g_0(\dot\sigma_\D, \dot\sigma_\D),\]
		where $\dot\sigma_\D$ is, for a.e. $s\in[0,1]$, the orthogonal projection of $\dot\sigma(s)$ on \D. Since the vector fields $\{X_i\}_{1\in \{1,\ldots ,d\}}$ generate the distribution $\D$ and \eqref{subRiem} holds, we have 
		\[g_0(\dot\sigma_\D, \dot\sigma_\D)\leq \sum_{i=1}^du^2_i.\]
		Hence, recalling that $-\omega(X_0)=\|\omega\|$ we get,  for $\dot\sigma\neq 0$,
		\bml\label{ineq1}\frac{g_0(\dot\sigma,\dot\sigma)^2}{\omega(\dot\sigma)^2}= \frac{g_0(\dot\sigma,\dot\sigma)^2}{\|\omega_\sigma\|^2u^2_0}\leq \frac{(u_0^2+\sum_{i=1}^du^2_i)^2}{\|\omega_\sigma\|^2u^2_0}\\\leq \frac{2u_0^4+2(\sum_{i=1}^du^2_i)^2}{\delta^2 u_0^2}
		=\frac{2}{\delta^2}\left(u_0^2+\frac{(\sum_{i=1}^du^2_i)^2}{u_0^2}\right).\eml
		Let then for each $\sigma\in j(f(K))$, $P_\sigma$ be a subdivision of the interval $[0,1]$	 according to the definition of the admissible controls; let us denote by  $J_h$ the intervals associated with  $P_\sigma$ where $u_0\neq 0$.
		Since on the possible existing intervals where $u_0=0$, we have that $u=0$ (and $\dot\sigma=0$ there),  and  we agree  to set the value of the Kropina metric $-g(v,v)/\omega(v)$ equal to $1$ at the $0$-section,  from \eqref{ineq1} we get:
		\beq\label{Esigma}E(\sigma)\leq  \inte \frac{g_0(\dot\sigma,\dot\sigma)^2}{\omega(\dot\sigma)^2} \de s\leq 1+\frac{2}{\delta^2}\left(\inte u_0^2\de s+\sum_{h}\int_{J_h} \frac{(\sum_{i=1}^du^2_i)^2}{u_0^2} \de s\right).\eeq
		The first integral term  in the right-hand side of  \eqref{Esigma}  is bounded on $j(f(K))$ since this is a compact set in the $W^{1,2}$-topology and $u_0^2=g_0(\dot\sigma, X_0)^2$. Let us show that  the sum of the integrals appearing in \eqref{Esigma} is bounded on $j(f(K))$ as well.
		Recalling again the definition of the admissible controls and in particular \eqref{Csquare},  we have
		\bmln \sum_{h}\int_{J_h} \frac{(\sum_{i=1}^du^2_i)^2}{u_0^2} \de s=\sum_{h}\int_{J_h} \frac{(\sum_{i=1}^d\alpha_{J_h i}^2\xi_{J_h}^2)^2}{\xi_{J_h}^4} d t\\=\sum_{h}\int_{J_h} (\sum_{i=1}^d\alpha_{J_h i}^2)^2 d t\leq \sum_h C^4|J_h|\leq C^4.\emln}
						\end{proof}
	
	\section{Multiple geodesics between  two points}\label{multiplesec}
	In this section we prove our main result about the existence of infinitely many  geodesics between two given points on a Randers-Kropina space $(S,F)$. As a consequence, we obtain multiplicity   results for   future-pointing lightlike geodesics in the spacetime associated with $(S,F)$ and for generalized solutions of the associated Zermelo's navigation problem. 
	
	The multiplicity result is obtained by applying Lusternik-Schnirelman theory to each energy functional of the approximating Randers metrics $F_\eps$. The idea to obtain multiple critical points of a functional from branches of families of critical points of approximating functionals is not new; it has been used in the so-called {\em penalization method} employed in the search of critical points of a functional defined on a non-complete   manifold (see e.g. \cite[Theorem 4.4.9]{Masiel94} for an example of application of this method to the geodesic connectedness problem of  an open 	 stationary Lorentzian manifold with boundary invariant by the flow of the timelike Killing vector field, or \cite{BaCaGS11} for an  application to the convexity of an open region with convex boundary in a Finsler manifold).
			
	\bt\label{multiplegeos}
	Let $(S, F)$ be a Randers-Kropina space and $x_0, \ x_1$ be  any  two points on $S$.   Let us assume that
	\begin{itemize}
		\item[(a)]  $S$ is   non-contractible;
		\item[(b)] condition \eqref{star} holds;
		\item[(c)]  \bb the distribution $\mathcal D:=\ker\omega$  is   nowhere integrable (i.e. $(\omega\wedge\de \omega)_p\neq 0$, for all $p\in S$) and  assumptions~\eqref{ass1}--\eqref{ass2} hold;\eb
		\item[(d)] whenever $x_0=x_1$, either $\Lambda(x_0)\neq 0$ or $\Lambda (x_0)=0$ and   $\de\Lambda_{x_0}(\mathcal D_{x_0})\neq 0$.
		\end{itemize} 
	 
	Then there exist infinitely many  geodesics $\sigma_h$ of $F$,  $h\in \N$,   connecting $x_0$ to $x_1$ and such that their $F$-lengths   diverge.
	\et	
	\begin{proof}
		By a result of E. Fadell and S. Husseini, \cite[Proposition 3.2]{FadHus91}, we know  from (a)  that $\cat\big(W^{1,2}_{x_0x_1}(S)\big)=+\infty$ and there exist compact subsets in $W^{1,2}_{x_0x_1}(S)$ with arbitrarily  large category. 
		Let then,  for each $k\in\N\setminus\{0\}$, $\Gamma(k):=\{A\subset W^{1,2}_{x_0x_1}(S):\cat_{W^{1,2}_{x_0x_1}(S)}(A)\geq k\}$  and let us denote by  $K_{k}$ a compact set in $\Gamma(k)$.  Since, for each $\eps>0$, the energy functional $E_\eps$ is a $C^{1,1}$ functional on $W^{1,2}_{x_0x_1}(S)$, bounded from below and,  thanks to assumption (b),  satisfying the Palais-Smale condition (recall Remark~\ref{aboutstar}) we know that for each $\eps\in(0,\bar\eps]$ and each $k$
		\[c_{\eps,k}:=\inf_{A\in \Gamma(k)}\sup_{\sigma\in A} E_\eps(\sigma)\] 
		is a critical value of $E_\eps$ (see e.g. \cite[Theorem 2.6.5]{Masiel94}). 
		 Then, for each $k$,   we have that $c_{\eps, k }\leq \max_{j(f(K_{ k }))} E_{\eps}$ for all $\eps\in (0,\bar 
		\eps]$ (recall that by Remark~\ref{catK}, $\cat_{W^{1,2}_{x_0x_1}(S)}j(f(K_{ k }))\geq  k $). 	Since  from Proposition~\ref{monotonicity} $E$ controls from above every $E_\eps$ (when restricted to admissible curves)  and  $E$ is bounded on  $j(f(K_k))$  by Proposition~\ref{bounded}, we have that, for each  $k$,  the family of critical values $c_{\eps, k  }$ is bounded  from above by $\sup_{j(f(K_{ k }))} E$. 
		On the other hand, for each $c>0$ the sublevels $E^c_\eps:=\{\sigma\in W^{1,2}_{x_0x_1}(S):E_\eps(\sigma)\leq c\}$ have finite category (see \cite[Proposition 2.6.10]{Masiel94}) and then,  for a fixed $\alpha\in\R$, there exists $k(\alpha)\in\N$ such that $A\cap \big(W^{1,2}_{x_0x_1}(S)\setminus E_\eps^\alpha\big)\neq \emptyset$ for all $A\in\Gamma(k(\alpha))$.    Hence, we can take $\alpha_1>\sup_{j(f(K_{ 1 }))} E$ and then $ k_2 >1 $ such that for all $A\in \Gamma(k_{ 2 })$, it holds $A\cap \big(W^{1,2}_{x_0x_1}(S)\setminus E_\eps^{\alpha_1}\big)\neq \emptyset$. This implies 
		that $\alpha_1\leq c_{\eps,k_{ 2 }}\leq \sup_{j(f(K_{k_{ 2 }}))} E$. Thus, by iteration, we can construct two strictly increasing sequences $\{\alpha_h\}$, $\alpha_h\to+\infty$, and $\{ k_h \}$, with $ k_1=1 $, such that  
		\beq\label{pinch}\sup_{j(f(K_{k_{ h }}))}E<\alpha_h\leq c_{\eps,k_{ h+1 }}\leq \sup_{j(f(K_{k_{ h+1 }}))} E.\eeq
		 Thus, for each $h$, we get a family of curves  $\sigma_{\eps,h}$ between $x_0$ and $x_1$  such that, for any $\eps\in (0,\bar\eps]$, $\sigma_{\eps,h}$ is a geodesic  of the Fermat metric $F_\eps$ and from \eqref{pinch} these geodesics have energies bounded from above. Hence,  Proposition~\ref{energies} gives   a  geodesic   $\sigma_h$ of the metric $F$  between $x_0$ and $x_1$, which by \eqref{pinch} satisfies that   $\alpha_{h-1}\leq E(\sigma_h)$, i.e.    $\sqrt{2\alpha_{h-1}}\leq \ell_{F}(\sigma_h)$,  and since $\alpha_{h}\to +\infty$ the proof is complete.
	\end{proof}
	Since a Kropina metric corresponds to the case of a Randers-Kropina one where $\Lambda =0$ (and $\omega_x$ does not vanish at any point $x\in S$),  Theorem~\ref{multiplegeos} applies to a Kropina space provided that $x_0\neq x_1$. 
	\begin{corollary}
		Let $(S, F)$ be a Kropina space and $x_0, \ x_1$ be two points on $S$ such that $x_0\neq x_1$.  Assume that condition \eqref{star} hold, $S$ is  non-contractible, \bb the distribution $\mathcal D:=\ker\omega$  is  nowhere integrable and assumptions~\eqref{ass1}--\eqref{ass2} hold. \eb
		Then there exist infinitely many  geodesics $\sigma_h$ of $F$  connecting $x_0$ to $x_1$ and such that their $F$-lengths   diverge.
	\end{corollary}
	Thanks to Theorem~\ref{fermat}, the geodesics obtained in Theorem~\ref{multiplegeos} correspond to   infinitely many future-pointing lightlike geodesics $\gamma_h=(\sigma_h,t_h)$ between an event  $(x_0,t_0)\in S\times\R$ and a flow line of $\partial_t$, passing through an event $(x_1,t_0)$, in a spacetime $(S\times \R, g)$, $g$ as in \eqref{g},  having arrival times $t_h(1)\to +\infty$. 
	
	\bere 
	We point out that the infinitely many geodesics obtained in Theorem~\ref{multiplegeos} can be interpreted as generalized solutions of a Zermelo's navigation problem, in the sense that they are time-travel stationary paths  rather than time-travel minimizers. Indeed, as mentioned in the introduction, given a Riemannian manifold $(S,g_0)$ and a vector field $W$ on $S$, such that $g_0(W,W)\leq 1$, we can associate a Randers-Kropina space $(S,F)$ to the data $(g_0, W)$, by taking $\omega=-g_0(\cdot, W)$ and $\Lambda=1-g_0(W,W)$ (see  \cite[Propositions 2.55 and 2.57]{CaJaSa14}).
	From Theorem~\ref{fermat}, being $t_h(1)=t_0+\int_{\sigma_h} F$, the geodesics between two different points on $S$ obtained in Theorem~\ref{multiplegeos} are critical points   of the arrival time functional $\sigma\in\Omega_{x_0,x_1}(\mathcal A)\mapsto T(\sigma)=t_0+\int_\sigma F$ (cf. \cite[Theorem 7.8]{CaJaSa14} for a general  statement about stationary points of the arrival time). 
	Notice that,  from \bb an extension of the  Chow-Rashewskii theorem for affine control systems (see \cite[\S 4.5]{Jurdje97}), \eb the nonintegrability assumption for the distribution orthogonal to $W$  ensures the existence of an admissible smooth curve between any two points in $S$, which means that the whole manifold $S$ is ``navigable'' starting from any point.
	\ere
	\bere\label{minimo}
	As a final remark, we notice   that if  the manifold $S$ is contractible then  $\cat\big(W^{1,2}_{x_0x_1}(S)\big) = 1$. In this case, the existence of at least one geodesic between $x_0$ and $x_1$, $x_0\neq x_1$,  follows   from \cite[Theorem 4.9-(i)]{CaJaSa14} provided that condition \eqref{star} holds  and \bb the distribution $\mathcal D$  is nowhere  integrable.  In fact, under the latter  assumption, by  an extension of the  Chow-Rashewskii theorem for affine control systems (see \cite[\S 4.5]{Jurdje97})), \eb  there exists an admissible curve between $x_0$ and $x_1$; moreover, from Remark~\ref{aboutstar}, the spacetimes $(S\times\R, g_\eps)$ are globally hyperbolic for each $0\leq \eps<\bar\eps$, and  then $(S\times\R, g)$  is also globally hyperbolic and \cite[Theorem 4.9-(i)]{CaJaSa14} applies.
	\ere

\end{document}